\documentclass[11pt,a4paper,
]{amsart}                 
\usepackage[T1]{fontenc}             
\usepackage[utf8]{inputenc}          \usepackage[english]{babel}       
\usepackage{graphicx}                      %
\usepackage[font=small]{quoting}            %
\usepackage{caption}  
\usepackage[colorlinks=true,linkcolor=blue]{hyperref}
\usepackage[top=1.2in,bottom=1.5in,left=1.2in,right=1.2in]{geometry}

\usepackage{verbatim}
\usepackage{color}
\usepackage{picture}
\usepackage{amsmath,amsthm,amsfonts,amssymb}
\usepackage{comment}
\usepackage{hyperref}
\hypersetup{colorlinks,linkcolor={blue},citecolor={blue},urlcolor={red}}

\newtheorem{thm}{Theorem}[section] 
\newtheorem{theorem}{Theorem}[section] 
 
\newtheorem{lemma}[thm]{Lemma}

\newtheorem{proposition}[thm]{Proposition} 
 
\newtheorem{defn}[thm]{Definition} 

		\newtheorem{remark}[thm]{Remark} 
\theoremstyle{remark}

\theoremstyle{definition}

\def\e{\epsilon}
\def\S{\Sigma} 
\def\n{\nabla}

\def\p{\partial}

\def\a{\alpha}

\def\n{\nabla}

\def\p{\partial}
\def\e{\epsilon}
\def\a{\alpha}

\def\g{\gamma}
\def\d{\delta}

\def\l{\lambda}

\def\n{\nabla}
\def\<{\langle}
\def\>{\rangle}
\def\div{{\rm div}}

\def\n{\nabla}

\def\p{\partial}
\def\e{\epsilon}
\def\a{\alpha}

\def\g{\gamma}
\def\d{\delta}
\def\l{\lambda}

\def\R{\mathbb{R}}

\def\ds{\displaystyle}

\def \vs{\vspace*{0.2cm}}

{\left\{\begin{array}{@{}l@{}}}{\end{array}\right.}
\patchcmd{\abstract}{\scshape\abstractname}{\textbf{\abstractname}}{}{}
\makeatletter 
\def\@makefnmark{} 
\makeatother

\numberwithin{equation}{section}
\numberwithin{exa}{section}

 \makeatletter
 \@namedef{subjclassname@2020}{\textup{2020} Mathematics Subject Classification}
 \makeatother
\begin{document}

\title[Log-Sobolev inequality]{ The Logarithmic Sobolev inequality on non-compact self-shrinkers
}

\author{Guofang Wang,  Chao Xia and Xiqiang Zhang}
\thanks{C.X. is supported by NSFC (Grant No. 12271449, 12126102) and the Natural Science Foundation of Fujian Province of China (Grant No. 2024J011008).}

\address{Mathematisches Institut, Albert-Ludwigs-Universit\"{a}t Freiburg, Freiburg im Breisgau, 79104, Germany}
\email{guofang.wang@math.uni-freiburg.de}

\address{School of Mathematical Sciences, Xiamen University, Xiamen, 361005, P. R. China}
\email{chaoxia@xmu.edu.cn}

\address{Mathematisches Institut, Albert-Ludwigs-Universit\"{a}t Freiburg, Freiburg im Breisgau, 79104, Germany}
\email{xiqiang.zhang@math.uni-freiburg.de}

\subjclass[2020]{Primary: 53A10 Secondary: 53C40,
	53A07, 35A23}

\keywords{ Log-Sobolev inequality, ABP method, self-shrinker}

\begin{abstract}

	In the paper we
	establish an optimal logarithmic Sobolev inequality for complete, non-compact, properly embedded self-shrinkers in the Euclidean space, which generalizes 
	a recent result of Brendle \cite{Brendle22}
	for closed self-shrinkers. 
	We first provide a  proof for the  logarithmic Sobolev inequality in the Euclidean space by using the Alexandrov-Bakelman-Pucci (ABP) method. 
	Then we use this approach to show 
	an optimal logarithmic Sobolev inequality for  complete, non-compact, properly embedded self-shrinkers in the Euclidean space, which is a sharp version of  the result of Ecker in \cite{Ecker}.
The proof is a noncompact modification of Brendle's proof for closed submanifolds and has a big potential to provide new inequalities in noncompact manifolds.
 

	\
	
	\noindent {\bf MSC 2020: }\\
	{\bf Keywords:}   Log-Sobolev inequality, ABP method, self-shrinker \\

\end{abstract}
 
 \maketitle

\section{Introduction}

Very recently, S. Brendle  \cite{Brendle21} used cleverly the ABP approach to establish optimal isoperimetric inequalities and also optimal Michael-Simon-Sobolev inequality for minimal submanifolds in
the Euclidean space, 
which solves a long-standing conjecture.  In \cite{Brendle22}, Brendle also obtained the following optimal logarithmic Sobolev inequality for closed (comapct without boundary) submanifolds in
the Euclidean space {of arbitrary dimension and codimension.

\

\noindent{\bf Theorem A} (\cite{Brendle22}) \textit{
    Let $\Sigma$ be an $n$-dimensional closed submanifold in $\R^{n+m}$. Then for any positive function $f$ on $\Sigma$ with $\int_\Sigma f dvol =1$, there holds 
\begin{equation}
    \label{eq0.1}
   \begin{array}{rcl} \ds\int_\Sigma \left\{\frac {|\n f|^2} f-f \log f +|H|^2 f\right\}  dvol \ge \ds\vs n+\frac n 2 \log(4\pi). 
    \end{array}    
\end{equation}  }

\

In particular, Theorem A implies an optimal  logarithmic Sobolev inequality for closed
self-shrinkers.

\

\noindent{\bf Theorem B} (\cite{Brendle22}) \textit{
    Let $\Sigma$ be an $n$-dimensional closed self-shrinker in $\R^{n+m}$.  Then
    for any positive function $\varphi$ with
    $\int_\S \varphi \, d\gamma=1$, where $d\gamma =\frac 1 {(4\pi)^\frac n 2} e^{-\frac {|x|^2} 4 }  dvol$, there holds
\begin{equation}\label{eq0.2}
     \int_\S \frac {|\n \varphi|^2}\varphi d\gamma
     \ge \int_\S \varphi \log\varphi \,d\gamma.
    \end{equation}
}

\

{A submanifold $\S\subset \R^{n+m}$ is called a \textit{ self-shrinker} if it satisfies 
$$H+\frac12x^\perp =0,$$ where $H$ is the mean curvature vector field of $\Sigma$ and $x^\perp$ is the normal component of $x$.} 

The logarithmic Sobolev inequality for submanifolds was first studied by Ecker in \cite{Ecker},
where he obtained a logarithmic Sobolev inequality with an extra term $c(n) \int \varphi \, d \gamma$.
See also the work of Gross  \cite{Gross91}, where the author established a logarithmic Sobolev inequality for submanifolds of certain probability spaces. 

One of our main results in this paper is to show that Theorem B is true for complete, non-compact properly embedded self-shrinkers.

\begin{theorem}\label{thm1}
	Let $\Sigma$ be an $n$-dimensional complete, non-compact properly embedded self-shrinker  in $\R^{n+m}$. Then for any positive function $\varphi$ on $\Sigma$ with $\int_\Sigma \varphi \,d \gamma =1$, there holds
	\begin{equation}
	\label{eq0.3}
	\int_\S \frac {|\n \varphi|^2}\varphi \,d\gamma \ge \int_\S \varphi \log\varphi \,d\gamma.
	\end{equation}
 Equivalently,   for any positive function $f$ on $\Sigma$ with $\int_\Sigma f dvol =1$, there holds
	\begin{equation}
	\label{eq0.4a}
	\begin{array}{rcl} \ds\int_\Sigma \left\{ \frac {|\n f|^2} f-f \log f +|H|^2 f\right\}dvol\ge \ds n+\frac n 2 \log(4\pi).
	\end{array}    
	\end{equation}
 
\end{theorem}

Since $\R^n$ is a self-shrinker, Theorem \ref{thm1} implies the Gaussian logarithmic Sobolev inequality in $\R^n$, due to A. J. Stam 
\cite{Stam59} in 1959 or
P. Federbush \cite{Federbush} in  1969,  often attributed to L. Gross. 
Gross showed in \cite{Gross75}  that 
the Logarithmic Sobolev inequality (short, LSI) is  equivalent to hypercontractivity of the semigroup of a diffusion process.
Since his work it	has been intensively studied and found rich connections to analysis, differential geometry, and probability (for surveys we refer to 
\cite{AB00, GZ03, BGL}).
For instance, the 	LSI  has played an important role in  stochastic analysis and also
played a crucial role in Perelman’s proof (2002) of the Poincaré
conjecture.

Inequality \eqref{eq0.3} is clearly optimal. In fact,  $\varphi=1$ achieves equality in \eqref{eq0.3} for $\S=\R^n$. 
We believe that equality in \eqref{eq0.3} is never achieved  on a  closed self-shrinker.
Actually we believe that the following question has an affirmative answer.

\

\noindent {\bf Question 1.} {\it Let  $\S^n\subset\R^{n+m}$ be a complete properly embedded  self-shrinker. If equality in \eqref{eq0.3} holds for some positive function $\varphi$, is $\S$ isometric to $\R^n$?}

\


There are many examples of compact or non-compact self shrinkers. The most important ones are the sphere and cylinders, ${\mathbb S}^k(\sqrt{2k})\times \R^{n-k}$, $k=0,1,2\cdots, n.$ Also most of known examples of non-compact self-shrinker are properly embedded. Here by a properly embedded submanifold it means that the  immersion $x:\Sigma\to \R^{n+m}$ is proper.  Thanks to a result of Ding-Xin \cite{Ding-Xin} and Cheng-Zhou \cite{Cheng-Zhou}, the properness of a self-shrinker is equivalent to the polynomial volume growth.

For a general complete, non-compact submanifolds in the Euclidean space we obtain a slightly weaker inequality than \eqref{eq0.1} in Theorem A.
\begin{theorem}\label{thm2}
	Let $\Sigma$ be an $n$-dimensional complete, non-compact submanifold in $\R^{n+m}$ with a  polynomial 
	 growth  condition (see Definition \ref{Def1} below). 
  Then for any positive function $f$ on $\Sigma$ with $\int_\Sigma f dvol =1$, there holds
	\begin{equation}
	\label{eq0.4}
	\begin{array}{rcl} &&\ds\int_\Sigma \left\{ \frac {|\n f|^2} f-f \log f +|H|^2 f\right\}dvol\\&\ge& \ds\vs n+\frac n 2 \log(4\pi)  \ds +\int_\Sigma f \left|\frac 12 x^\perp + H\right|^2 dvol -\log \int_{\Sigma}f (x) e^{\left|\frac 12 x^\perp + H\right|^2} dvol.
	\end{array}    
	\end{equation}
\end{theorem}
Theorem \ref{thm1} follows clearly from  Theorem \ref{thm2}.

{ To indicate the difference between the compact case and the non-compact case, let us first  briefly recall the  ABP method of Brendle to achieve \eqref{eq0.1}. Under the normalization 
$\int_\Sigma f=1$, one considers 
\begin{equation}\label{ABP-equation}
    \div(f\nabla u)=f \log f -\frac {|\n f|^2} f -|H|^2 f +\a f.
\end{equation}
(Here we have used a different normalization and its corresponding equation other than the one in \cite{Brendle22} in our convenience. But there is no any difference.)
 Using the map  $$\Phi(x,y)=\n u(x)+y: 
 \{(x,y) \, |\, x\in \Sigma, y\in T_x^\perp \Sigma, \n ^ 2 u(x)-\<II(x), y\> \ge 0\} \to \R^ {n+m}
 $$ and the area formula, he obtained 
 $$ \a \ge n+\frac n 2 \log(4\pi).$$ It follows, together with \eqref{ABP-equation} and the divergence theorem,
 inequality \eqref{eq0.1}. The closedness of $\Sigma$ is used in the proof of the existence of \eqref{ABP-equation}, the surjectivity of $\Phi$ and the divergence theorem, $\int_\S \div (f \n u) =0.$ In other words, for a non-compact case we need to take care of these 3 points. %
 
 We are not able to prove the existence of solutions to \eqref{ABP-equation}, when $\Sigma$ is non-compact. But we are able to prove the existence
 of solutions to  the following equation 
\begin{equation}\label{eq0.6}
\div(f\nabla u) = f\log f
	-\frac {|\n f|^2}{f}-|H|^2 f+ \left|\frac 12 x^\perp + H\right|^2f  +\a f
\end{equation}
 for suitable functions $f$ and under a polynomial  growth condition. This will be shown enough to establish our inequalities.
 Our approach will crucially rely on an existence of  a {\it canonical pair} $f_0$ and $u_0$ below.



\medskip

\noindent{\bf Observation}. {\it On any $n$-dimensional submanifold $\Sigma\subset\R^{n+m}$, 
	the function $u_0=\frac 12 |x|^2 $ satisfies
	\begin{equation}
	\label{eq0.5}
	\div(f_0\nabla u_0) =
	f_0 \log f_0 - \frac {|\n f_0|^2}{f_0}   + \left|\frac 12 x^\perp  
 + H\right|^2f_0 -|H|^2f_0 +\a_0f_0,
	\end{equation}
 where $f_0=\frac 1 {(4\pi)^\frac n 2} e^{-\frac {|x|^2} 4 } $ and $\a_0=n+\frac n 2 \log (4\pi).$
}

\

In view of \eqref{eq0.5}, we consider \eqref{eq0.6} 
instead of  equation \eqref{ABP-equation}. 
It is easy to show that in order to show  inequality \eqref{eq0.4} we only need to consider functions with compact support. Let $f$ be such a function.
We perturb $f$ by using $f_0$  $$f_\e=\frac{f+\e f_0}{1+\e}, \quad \e>0.$$ 
Then we prove there exists a solution $u_\e$ to \eqref{eq0.6} on $\S$ for some $\a_\e$, provided that $\S$ satisfies the polynomial  growth condition.
Moreover, We show that the solution $u_\e$ is of sub-linear growth and satisfying $\div(f_\e \n u_\e)=0.$
The sub-linear growth implies the surjectivity of $\Phi$. Such $u_\e$
enables us to apply the ABP method to achieve Theorem \ref{thm2}.
}

We remark that with \eqref{eq0.6} we can only obtain \eqref{eq0.4}, instead of \eqref{eq0.1}.  It is clear, when $\Sigma$ is self-shrinker, which just means that the extra  term 
$\left|\frac 12 x^\perp + H\right|^2$ vanishes, we have the optimal form \eqref{eq0.3}
for properly embedded self shrinkers, together with  the fact that properly embeddedness of 
a self-shrinker is equivalent to the polynomial volume growth, by \cite{Ding-Xin} and \cite{Cheng-Zhou}. It  remains open 
if Theorem \ref{thm1} is true for any non-compact submanifolds under the polynomial  growth condition.

We make a few comments on the proof of the logarithmic Sobolev inequality in the Euclidean space.  There are at least  \textit{``More than fifteen proofs of the 	logarithmic Sobolev inequality''} written by Ledeux as a title of his survey \cite{Ledeux_15}, where the author divides the proofs into two groups:  one is the proofs for hypercontractivity, and another is for the proofs of the LSI. 
Here we just mention more analytic methods in  the second group:  Deduction of the LSI 
from the classical Euclidean	isoperimetric inequality or the Gaussian isoperimetric inequality  by the coarea formula \cite{Rothaus85} and \cite{Ledeux85}, deduction from the optimal Sobolev inequality  \cite{Beckner99} and \cite{DPD02, DPD03},  a proof using  the Prekopa-Leindler inequality which is  a functional form of the Brunn-Minkowski inequality and an approach of optimal transport
\cite{Otto-Villani00} and \cite{Cordero-Erausquin02}. A possible simplest proof  suggested in \cite{Ledeux85} is a semi-group theoretical proof of D. Bakry and M. Émery in \cite{Bakry-Emery85} by studying the second derivative of the entropy along the Ornstein-Uhlenbeck semigroup. Our result provides a new proof of it.

Our approach is a non-compact modification of the ABP method of Brendle \cite{Brendle22}. {To the best of our knowledge, this is the first attempt to apply the ABP method in the non-compact setting}.   There are many generalizations of his result to closed submanifolds in other ambient spaces, see \cite{Pham23, Dong24, Lee-Ricc24, Yi-Zheng24}. Also the ABP method is successful to be used to obtain
the optimal isoperimetric inequalities for closed domains in  minimal submanifolds
in \cite{Brendle21} as mentioned at the beginning of the Introduction. See also the previous work of Cabr\'e and  Cabr\'e-Ros-Oton-Serra \cite{Cabre08, CRS16}, and also Trudinger \cite{Trudinger94}. The ABP method for isoperimetric inequalities can be applied to  cases with free boundary or even capillary boundary, see \cite{LWW} and \cite{Fusco}

It sounds that the approach presented in the paper is more flexible and powerful. It provides not only also a proof for the optimal Euclidean  $L^p$-log-Sobolev inequality ($p>1$) 
and but also establish  a new Logarithmic Sobolev inequality for positive definite symmetric 2-tensors on the Euclidean space, which could  be viewed as a non-commutative version of the LSI. 
The key point is the existence of such a canonical pair as in the above Observation.
We will  carry them out in a forthcoming paper.

In the paper, all functions $f$ or $\varphi$ are non-negative with finite integrals 
integrated over the support of $f$ or $\varphi.$

\

The paper is organized as follows: In Section 2 we prove the existence of a solution of \eqref{eq0.6} for $f_\e$ in the case of  $\Sigma =\R^n$ and then provide 
an ABP proof for the classical logarithmic Sobolev inequality \eqref{eq1.1} in the Euclidean space. 
In Section 3 we modify the ideas presented in Section 2  to prove Theorem \ref{thm1} and Theorem \ref{thm2}. Some discussions related to the entropy introduced in \cite{Colding-Minicozzi}
are given in the end of the paper.

\section{A proof of the classical LSI}

In order to illustrate our ideas, we consider first  the case of the Euclidean space.
Namely in this section we will   provide a new proof of the classical logarithmic Sobolev inequality

\begin{equation}\label{eq1.1}
\int _{\R^n} \frac {|\n f|^2} fdx-
\int_ {\R^n} f \log f dx \ge  n +\frac  n2 \log (4\pi) 
\end{equation}
for any non-negative function $f$ with normalization that $\int_{\R^n } f dx=1$, which is equivalent to 
the Gaussian logarithmic Sobolev inequality
\begin{equation} \label{eq1.2} 
\int_{ \R^n } \frac {|\n \varphi|^2} \varphi \,d\gamma \ge \int_{\R^n} \varphi \log \varphi \, d \gamma ,
\end{equation} for any non-negative function $\varphi$ with $ \int_{\R^n} \varphi \, d\gamma=1$,
where   $d\gamma = d\gamma_n=\gamma dx$ be the  Gaussian probability measure on $\R^n$ with density $\gamma= \gamma_n=\frac 1{(4\pi)^{\frac n 2}} e^{-\frac {|x|^2 } 4}$. 
It is easy to check that if we insert $f=\varphi\gamma $ into \eqref{eq1.1}, then we have \eqref{eq1.2}.

For a positive function $f$ with normalization $\int_{ \R^n }f dx=1$, we want to find a pair of function $u$ and number $\a$ satisfying 
\begin{equation}
\label{P-1}
\div (f \nabla u)=f\log f-\frac{|\nabla f|^2}{f} +\alpha f
\end{equation}
with a superlinear growth condition
\begin{equation}
\label{P-2}
\frac{u(x)}{|x|}\to \infty, \qquad \hbox{ as } |x|\to \infty
\end{equation}
and  
\begin{equation}
\label{P-4}
\int_{\mathbb R^n} \div (f\nabla u)\,dx=0.
\end{equation} 
If we have a solution $(u,\a)$ of  \eqref{P-1} satisfying \eqref{P-2}, then we can follow exactly  the idea of Brendle to show
$$\a \ge n +\frac n 2 \log (4\pi).$$ See below. If \eqref{P-4} is also true, then we have
$$\int_{ \R^n } \left\{\frac{|\nabla f|^2}{f} -f\log f \right\} \, dx = \a\int_{ \R^n }f =\a \ge n +\frac n 2 \log (4\pi),$$
the logarithmic Sobolev inequality. However we are unable to show the existence of $(u,\a)$ to  \eqref{P-1}--\eqref{P-4} for {\it all } non-negative $f$. In fact we doubt that this is true for all $f$.

It is easy to see  by exhaustion that in order to show \eqref{eq1.1} for all non-negative (or positive) functions we only need to show it for all non-negative functions with compact support.
Let $f$ be a non-negative function with compact support $K$. For such a $f$  we are still 
unable to show the existence of $(u,\a)$ with \eqref{P-1}--\eqref{P-4}. The crucial observation is a simple fact that  there is a solution $(u_0,\a_0)$  of  \eqref{P-1}--\eqref{P-4} for $f_0=\gamma$
\begin{equation}
u_0=\frac 12|x|^2, \qquad \a_0=n+\frac n 2 \log (4\pi).
\end{equation}

\begin{lemma}
	\label{L-1}
	For $f_0=\frac{1}{(4\pi)^{n/2}}e^{-\frac{|x|^2}{4}}$ and $u_0=\frac{|x|^2}{2}$, we have
	\begin{equation}
	\label{L-2}
	\div(f_0\nabla u_0) = f_0 \log f_0-
	\frac {|\n f_0|^2}{f_0} +\a_0f_0,
	\end{equation}
	with
	$
	\alpha_0=n+\frac{n}{2}\log(4\pi).$ Moreover \eqref{P-2} and \eqref{P-4} are satisfied.
\end{lemma}

\begin{proof}
	Just a  direct calculation.
\end{proof}

Using $f_0$ we consider the following perturbation of $f$
\begin{equation}
f_\e =\frac 1 {1+\e} (f+\e f_0),  \qquad \forall \e>0.
\end{equation}
For  these $f_\e$ we now can  show the existence.
\begin{proposition} \label{prop1} 
	Let $f$ be a non-negative function with compact support $K$. Then for any $\e>0$ there is a solution $(u_\e, \a_\e)$ of \eqref{P-1}--\eqref{P-4} for $f_\e$.
\end{proposition}
For $f_\e$ we then can follow idea of Brendle to prove $\a_\e\ge n+\frac n 2  \log(4\pi)$, which implies the LSI for $f_\e$. By taking $\e\to 0$ we obtain the LSI for any non-negative function $f$ with compact support, and hence for any non-negative functions on $\R^n$. This provides a new proof of inequality \eqref{eq1.1}.

\

\noindent{\it Proof of Proposition \ref{prop1}.}
Now, we start to prove  Proposition \ref{prop1}. We aim 
to find a solution $(u_\e, \a_\e)$ of \eqref{P-1}, that is,
\begin{equation}
\label{P-5}
\Delta u_\epsilon+\nabla \log f_\epsilon \cdot \nabla u_\epsilon  =\log f_\epsilon-|\nabla\log f_\epsilon|^2+\alpha_\epsilon.
\end{equation} 
with properties \eqref{P-2} and \eqref{P-4}. It is easy to check that it can reduced to find a solution
$(w_\e, c_\e)$ to 
\begin{equation}
\label{P-6}
\Delta w+\nabla \log g \cdot \nabla w+\ell(x)=c ~~~ \text{in} ~~\mathbb R^n     
\end{equation}
where $g=f_\e$ and 
\begin{equation}
\ell(x):=
|\nabla \log f_\epsilon|^2-|\nabla \log f_0|^2+\nabla \log \frac{f_\epsilon}{f_0}\cdot x+\log \frac{f_0}{f_\epsilon}
\end{equation}
with sub-quadratic growth
\begin{equation}
\label{P-10}
\lim_{|x|\to +\infty}\frac{w_\e(x)}{|x|^2}=0.
\end{equation}
In fact, with such a solution $(w_\e, c_\e)$ we have a solution of $u_\e=u_0+w_\e$ of \eqref{P-1} with $\a_\e=\a_0+c_\e$. \eqref{P-10}  implies the crucial property \eqref{P-2}. The property \eqref{P-4} can be also proved not  difficultly.

\begin{proposition}
	\label{prop2} 
	\eqref{P-6} has a solution $(w_\e, c_\e)$ with \eqref{P-10}.
\end{proposition}

\begin{proof}
	
	Set $$Lw:=L_gw:= \Delta w+\nabla \log g \cdot \nabla w.$$ The idea of proof is similar to the one in  \cite[Theroem 6.1]{Bardi-2016}. We will divide into several steps to prove the result.

	\

	\noindent	\textbf{Step 1.}  For every $\delta>0$, the equation
	\begin{equation}
	\label{P-11}
	\delta w-(Lw+ \ell)=0~~~ \text{in} ~{\mathbb R^n}
	\end{equation}
	has a bounded solution $w_\delta$. Moreover  for every $h\in(0,1]$ there exists $R_h>0$ independent of $\d$ such that
	\begin{equation}
	\label{P-12}
	-h\frac{|x|^2}{2}+\min_{|x|\leq R_h}w_\delta\leq w_\delta \leq \max_{|x|\leq R_h}w_\delta+h\frac{|x|^2}{2}.
	\end{equation}
	
	One can follow the Perron method or the sub/supersolution method to obtain a solution of \eqref{P-11} for $\d>0$.
	Since $\ell(x) = \log \frac \e {1+\e}$ outside $K$, $\|\ell\|_\infty$ is bounded for any fixed $\e>0$. It is clear that $\d^ {-1} \|\ell\|_\infty$ is a supersolution of \eqref{P-11}, while $-\d^ {-1} \|\ell\|_\infty$ a subsolution.
	Consider the set $S$ of all subsolutions of \eqref{P-11}  between $-\d^ {-1} \|\ell\|_\infty$ 
	and $\d^ {-1} \|\ell\|_\infty$ and set 
	$w_\d(x):=\sup \{ \varphi (x) ,\, \varphi\in S\}$. Notice that the maximum principle holds true for \eqref{P-11} for any $\d>0$. Hence by the
	Perron method  we obtain a solution $w_\d$ with
	\begin{equation}
	\label{P-13}
	\|w_\delta\|_{\infty}\leq \frac{1}{\delta}\|\ell\|_{\infty}.
	\end{equation} 
	(\ref{P-12}) follows from the maximum principle as follows.  For any 
	fixed $h\in (0,1]$, 
	and choose $R_h>0$  later, such that $K\subset B_{R_h}$. 
	Define  \begin{equation}
	v(x): =h  u_0+\max_{|y|\leq R_h}w_\delta.\nonumber
	\end{equation} One can check that it is a supersolution to (\ref{P-11}) in $|x|>R_h$. In fact, due to (\ref{P-13}), we have for 
	$|x|>R_h$,
	\begin{equation}
	\label{P-15}
	\begin{aligned}
	\delta v-(Lv+\ell) &=\delta h \frac{|x|^2}{2}+\delta  \max_{|y|\leq R_h}w_{\delta}-(h Lu_0+\ell) \\
	&\geq \delta h \frac{|x|^2}{2}  +\delta \max_{|y|\leq R_h}w_{\delta} -h\left( n-\frac{|x|^2}{2}\right) -\|\ell\|_{\infty} \\
	&{\geq \delta h \frac{|x|^2}{2}   -h\left( n-\frac{|x|^2}{2}\right) -2\|\ell\|_{\infty}} \\
	&\geq 0,
	\end{aligned}
	\end{equation}
	for large $R_h$, which is independent of $
	\delta$. {Clearly $v\ge w_\delta$ on $\p B_{R_h}$ and  $\lim\limits_{|x|\to\infty} (v-w_\delta)\ge 0$. }
	Then the maximum principle implies that $w_\d \le v$ in $\R^n\setminus B_{R_h}$, and hence $w_\d \le v$ in $\R^ n$. Another inequality in \eqref{P-12} holds similarly.

	\
	
	\noindent	\textbf{Step 2.}  We claim that  $v_\delta=w_\delta-w_\delta(0)$ are uniformly bounded in every compact set $A$.

	If not, there exists a compact set $A$ satisfying $A\supset \{x|~|x|<R_1\}$  such that 
	\begin{equation}
	\epsilon^{-1}_{\delta}:=\|v_{\delta}\|_{L^{\infty}(A)}\to +\infty.\nonumber   
	\end{equation}
	Set  $\psi_\delta=\epsilon_{\delta}v_{\delta}$. Then we have $\|\psi_\delta\|_{L^{\infty}(A)}=1$ and $\psi_\delta(0)=0$.
	From Step 1 we know	for any $x$\begin{equation}
	\label{P-16}
	\begin{aligned}
	\psi_\delta(x)&=\epsilon_{\delta}v_{\delta}=\frac{w_\delta(x)-w_\delta(0)}{\|v_\delta\|_{L^\infty(A)}}\\
	&\leq \frac{\frac{|x|^2}{2}+\max_{|x|\leq R_1}(w_\delta-w_\delta(0))}{\|w_\delta-w_\delta(0)\|_{L^{\infty}(A)}}\\
	&\leq 1+\epsilon_{\delta}\frac{|x|^2}{2}.\\ 
	\end{aligned}   
	\end{equation}
	Similarly, we also obtain
	\begin{equation}\label{P-16a}
	\psi_\delta(x)\geq -1-\epsilon_{\delta}\frac{|x|^2}{2}.
	\end{equation}
	Thus the sequence $\psi_\delta(x)$ is uniformly bounded in every compact subset in $\mathbb R^n$. 
	Since $\psi_\delta(x)$ solves the equation
	\begin{equation}
	\delta \psi_\delta+\delta \epsilon_\delta w_\delta(0)-\Delta \psi_\delta-\nabla \log g \cdot \nabla \psi_\delta-\epsilon_\delta \ell(x)=0 ~~\text{in} ~~ {\mathbb R^n}, \nonumber    
	\end{equation}
	and $\ell$ and  $\delta w_\delta(0)$ are bounded, we apply  the theory of linear elliptic equations to show that $\psi_\delta$ subconverges to $\psi$ in every compact set of $\mathbb R^n$. 
	It is easy to check that  $\psi$ satisfies 	\begin{equation}
	-\Delta \psi-\nabla \log g \cdot \nabla \psi=0  \qquad \text{in } {\mathbb R^n}\nonumber
	\end{equation}
	and 	 $\|\psi\|_{L^{\infty}(A)}=1$. On the other hand,  (\ref{P-16})-\eqref{P-16a} imply  that $|\psi|\leq 1$ in $ x\in {\mathbb R^n}\backslash A $. This implies that $\psi$ attains either a global maximum or a global minimum in $A$. Hence it must be a constant, which is either $1$ or $-1$,  by the strong maximum  principle. This contradicts  $\psi(0)=0$.
	
	\

 \
	
	\noindent\textbf{Step 3.}  Existence of solution  $w$ to (\ref{P-6}).
	
	From (\ref{P-13}), we know that $\delta w_\delta(0)$ is bounded, and hence  $\delta w_\delta(0) \to -c$ as $\delta\to 0$, by taking possibly a subsequence. 
	By Step 2, $v_\delta$ are uniformly bounded in every compact subset in  $\mathbb R^n$.
	Again the standard theory implies that $v_\delta$ (sub)converges to a $w$, which is a solution of \eqref{P-6}.
	
	\
	
	\noindent\textbf{Step 4.} The growth condition of $w$.
	
	Since (\ref{P-12}) is  independent of $\delta$ and it also holds for $v_\d:=w_\d-w_\d(0)$, we have that for any $h\in(0,1]$ there exists $R_h$ such that
	\begin{equation}
	-h\frac{|x|^2}{2}+\min_{|x|\leq R_h}w\leq w \leq \max_{|x|\leq R_h}w+h\frac{|x|^2}{2},\nonumber 
	\end{equation}
	which implies \eqref{P-10}. This finishes the proof of the Proposition.
\end{proof}



We remark that one can show the uniqueness of $w$ and $c$ of \eqref{P-6} with \eqref{P-10}. 

\

Now we  show Proposition \ref{prop1}.

\

\noindent{\it Proof of Proposition \ref{prop1}.} Let $u_\e=w_\e+u_0$. It is clear that $u_\e$ satisfies \eqref{P-1} and \eqref{P-2}. \eqref{P-4} is proved in the following Lemma.  \qed

\begin{lemma}\label{lem2.4}
	We have
	\[
	\int_{\R^n} \div(f_\e\n u_\e) \, dx=0. \]
\end{lemma}
\begin{proof} For simplicity of notation we omit the subscript $\e$. 
	From the construction of $w$, it is clear  that 
	\begin{eqnarray}\label{u-quad}|u(x)|  \le C(1+|x|^2)\end{eqnarray} for some $C>0$, while $f$ is exponentially decay. 
	Let $\xi$ be a cut-off function with support in $B_{2r}$ satisfying $\xi=1$ in $B_r$ and $|\n \xi(x)|\le 2 /|x|$ in $B_{2r}\backslash B_r$.
	It is easy to check that
	\begin{equation*}
	\begin{array}{rcl} \ds 
	\left| \int _{B_r} \div (f\n u) dx\right|&\le & \ds
	\left| \int _{\R^n} \div (f\n u) \xi dx\right|
	+	\left| \int _{\R^n\backslash B_r} \div (f\n u) \xi dx\right|.
	\end{array}
	\end{equation*}
	Using Equation \eqref{P-1}  and the growth of $f$ and $u$, it is easy to see that the second term in the right hand side tends to zero as $r\to \infty$. The first term is equal to
	\[\begin{array}{rcl}
	\ds \left| \int _{\R^n}  f\n u\n\xi dx\right| & \le & \ds\vs
	\frac 1 r \int_{B_{2r}\backslash B_r} f|\n u| dx \\
	&\le& \ds \frac 1 r  \left(\int_{B_{2r}\backslash B_r} fdx\right)^{\frac 12}
	\left( \int_{B_{2r}\backslash B_r} f|\n u|^2 dx\right)^{\frac 12}.\end{array}
	\]Hence the Lemma follows, if 
	$\int_{B_{r}} f|\n u|^2 dx$
	is uniformly bounded. This is true by a standard argument as follows
	{
		\begin{eqnarray*}
   \int_{\R^n }\xi^2|\n u|^2  fdx&=&-\int_{\R^n} u\xi^2\div(f\n u)- \int_{\R^n} 2\xi u\n\xi\n u f 
			\\ & \le & \left|\int_{\R^n} u\xi^2{\rm div}(f\n u)\right|+  \frac 12 \int_{\R^n} \xi^2 |\n u|^2 f+ 8\int_{\R^n} u^2 |\n \xi|^2 f.
		\end{eqnarray*}
		Hence
		\begin{eqnarray*}
  \int_{B_{r}} |\n u|^2  fdx  \le 
	\int_{\R^n }\xi^2 |\n u|^2  fdx        \le
			2\left|\int_{\R^n} u\xi^2{\rm div}(f\n u)\right| + 16  \int_{\R^n} u^2 |\n \xi|^2 f.
		\end{eqnarray*}
		From Equation \eqref{P-1} and \eqref{u-quad}, we see that 
		$\left|\int_{\R^n} u\xi^2{\rm div}(f\n u)\right|$ and $\int_{\R^n} u^2 |\n \xi|^2 f$ are both uniformly bounded. Thus $\int_{B_{r}} f|\n u|^2 dx$
		is uniformly bounded. The proof is completed.
	}
\end{proof}

\begin{remark}\label{rmk2.4}
It is easy to see that the above argument works for a similar equation as \eqref{P-1} on a complete manifolds with a polynomial volume growth.
    
\end{remark}

Now we can prove the LSI in the Euclidean space.

\

\noindent{\it Proof of logarithmic Sobolev inequality \eqref{eq1.1}}. With Proposition \ref{prop1} we can follow the idea in \cite{Brendle22} to 
give a proof of \eqref{eq1.1}.  Let $f_\e$ be the solution in Proposition \ref{prop1}. We  need to show that
\begin{equation}\label{25}
\a_\e \ge n+\frac n 2 \log (4\pi).
\end{equation}

Let us denote $u_\e$ by $u$, and define 
\begin{equation} \label{u}
U:=\{x\in \mathbb R^n: \n^2u (x)\geq0\},
\end{equation}
and a map $\Phi:\mathbb R^n \to \mathbb R^n$ by
\begin{equation}\label{Phi}
\Phi(x)=\nabla u(x)
\end{equation}
for any $x\in \mathbb R^n$.
\begin{lemma}
	\label{L-3}
	We have $\Phi(U)= \mathbb R^n$.
\end{lemma}
\begin{proof}
	For any $\xi\in \mathbb R^n$, consider $q(x)=u(x)-x\cdot\xi$. 
	It is crucial that \eqref{P-2} implies that
	$q$ attains a global minimum  at $\bar{x}  \in \mathbb R^n$. It is then easy to see that $\bar x \in U$ and $\Phi(\bar x)=\xi$, and hence $\Phi(U)=\R^n$.
\end{proof}

\begin{lemma}
	We have that for any $x\in U$ the Jacobian determinant of $\Phi$  satisfies
	$$0\leq e^{-\frac{|\Phi(x)|^2}{4}}\det \n\Phi(x)\leq f_\epsilon (x)e^{-n+\alpha_\epsilon}.$$
\end{lemma}

\begin{proof} The proof is the same as in \cite{Brendle22}. We do not repeat it here, but refer to the proof of Lemma \ref{lem3.2} below.
\end{proof}

Now \eqref{25} follows from
\begin{equation}
\begin{aligned}
1&=(4\pi)^{-\frac{n}{2}}\int_{\mathbb R^n}e^{-\frac{|\xi|^2}{4}}d\xi\\
&\leq (4\pi)^{-\frac{n}{2}}\int_{\mathbb R^n}e^{-\frac{|\Phi(x)|^2}{4}}|\det \n\Phi(x)|1_{U}(x)dx\\
&\leq (4\pi)^{-\frac{n}{2}}\int_{\mathbb R^n}f_\epsilon (x)e^{-n+\alpha_\epsilon}dx
=(4\pi)^{-\frac{n}{2}}e^{-n+\alpha_\epsilon}.  \nonumber
\end{aligned}
\end{equation} 
Inequality \eqref{25}, together with \eqref{P-1} and (\ref{P-4}), implies
\begin{equation}
\int_{\mathbb R^n}\frac{|\nabla f_\epsilon|^2}{f_\epsilon}dx-\int_{\mathbb R^n} f_\epsilon\log f_\epsilon dx =\a_\e \int_{\R^n} f_\e dx \ge n+\frac{n}{2}\log(4\pi) .\nonumber    
\end{equation}
By the dominated convergence, letting $\epsilon \to 0$, we have 
\begin{equation}
\int_{\R ^n}\frac{|\nabla f|^2}{f}dx-\int_{\R^n} f\log f dx\geq n+\frac{n}{2}\log(4\pi) ,\nonumber    
\end{equation}
for any non-negative $f$ with compact support and $\int_{\R^n} f dx=1$. This finishes the proof of \eqref{eq1.1}. \qed


\section{LSI for noncompact self-schrinkers}\label{sec3}

Let $\S^n\subset \R^{n+m}$ be an  $n$-dimensional complete, non-compact Riemannian submanifold with the induced metric $g$ from $\R^{n+m}$. In this section, we use  $\n$, $\n^2$, $\Delta$ and $dvol$ to denote the Levi-Civita connection, the Hessian, the  Laplacian and the volume form with respect to $g$ respectively. For a vector field $Y$ on $\R^{n+m}$, we use $Y^\top$  and $Y^\perp$ to denote the component of $Y$ that is tangential to $\S$ and normal to $\S$ respectively. Let $II$ and $H$ be the second fundamental form and the mean curvature vector field respectively.

Consider on $\R^{n+m}$ the function $$f_0(x):=(4\pi)^{-\frac n 2} e^{-\frac {|x|^2} 4},$$ which is equal to $\gamma_n$ if we restrict it to an $n$-dimensional subspace, and the function $$u_0:=\frac 12 |x|^2.$$
We denote the restriction of $f_0$ and $u_0$ on $\Sigma$ still by $f_0$ and $u_0$ respectively. { In this section we shall also use the volume form $$d\g=\g_0 dvol=(4\pi)^{-\frac n 2} e^{-\frac {|x|^2} 4}dvol$$ on $\S$. Notice that it is not the restriction of the Gaussian measure in $\R^{n+m}$ to $\S$.}

We first have the following simple, but as above crucial Lemma.

\begin{lemma}\label{lem5.1}Let $u_0:\Sigma\to \R$ defined by
	$u_0:=\frac 12 |x|^2$. We have that $u_0$ satisfies
	\begin{equation}\label{eq5.1}
	\div(f_0\nabla u_0) =f_0 \log f_0 -
	\frac {|\n f_0|^2}{f_0}+ \left|\frac 12 x^\perp + H\right|^2f_0 -|H|^2f_0 +\a_0f_0,
	\end{equation}
	where  
	$\a_0=n+\frac n 2 \log(4\pi)$.
\end{lemma}
\begin{proof}
	See the proof for example in \cite{Colding-Minicozzi}. For the convenience of the reader and also fixing the notations, we give the computation briefly. 
 \eqref{eq5.1} is equivalent to
 \begin{equation}\label{eq5.1a}
	f_0^{-1}\div(f_0\nabla u_0) = \log f_0 -
	|\n \log f_0|^2+ \left|\frac 12 x^\perp + H\right|^2 -|H|^2 +\a_0.
	\end{equation}
 It is easy to see $$\n u_0=x^\top, \quad \n \log f_0 =-\frac 12 x^\top, \quad \Delta u_0=n+\<x^\perp, H\>.$$ Hence the left hand side of \eqref{eq5.1a} is $n+\<x^\perp, H\>-\frac 12|x^\top|^2$, while the right hand side is
    \[\begin{array}
        {rcl} 
   &&\ds\vs -\frac n 2 \log (4\pi) -\frac 14 |x|^2 - \frac 14 |x^\top|^2 +\frac 14 |x^\bot |^2 + \<x^\perp, H\> +\a_0 \\
    &=& \ds  n  +\<x^\bot, H\> -\frac 12 |x^\top|^2.\end{array}
    \]
    The proof is completed.
\end{proof}

As mentioned in the Introduction, in view of \eqref{eq5.1} we consider the following
equation
\begin{equation}\label{eq5.2}
\div(f\nabla u) =f \log f-
	\frac {|\n f|^2}{f} + \left|\frac 12 x^\perp + H\right|^2f -|H|^2f +\a f,
\end{equation}
which is slightly different to the one considered in \cite{Brendle22}. As in Section 2, we attempt to find a solution pair
$(u,\a)$ of \eqref{eq5.2} with a  suitable growth condition, which we define.

\begin{defn}[Polynomial growth] \label{Def1} Let $\S\subset \R^{n+m}$ be a submanifold with induced metric $g$. We call $\S$ admitting a polynomial growth condition, if $\Sigma$ 
    has at most polynomial volume growth and its mean curvature vector field $H$ has at most polynomial volume growth, i.e.
    \begin{equation} \label{con}
vol_g(\S\cap B(r))\le c r^k \end{equation} and  \[ |H|\le c r^l,
          \]
    for some integers  $k,l$ and $c$, which is independent  of $r$. Here $B(r)$ is a ball in $\R^{n+m}$ of radius $r$.
\end{defn}
 \eqref{con} is called extrinsic polynomial volume growth condition.
Here for the simplicity we just use the extrinsic volume growth condition. In fact the intrinsic polynomial volume condition, which is weaker than the extrinsic one, is enough.

As discussed above, in order to obtain a logarithmic Sobolev inequality on $\Sigma$ we only need to consider any positive function $f:\Sigma \to\mathbb{R}_+$ with compact support and with the normalization
\[\int_\Sigma f=1.\] For $\e>0$, let $f_\e$ be defined by
\[
f_\e=\frac 1{1+\e} (f+\e f_0).
\] The aim is to find a solution $u$ of \eqref{eq5.2} with quadratic growth, for $f_\epsilon$. As above we only need to find a solution of
\begin{equation}
\label{eq5.3}
\Delta w + \n \log g \cdot \n w +\ell =c,
\end{equation}
with a decay
\begin{equation}
\lim_{|x| \to \infty}\frac {w(x)}{|x|^2}=0.
\end{equation}
Equation \eqref{eq5.3} is the difference of between Equation \eqref{eq5.1} and Equation \eqref{eq5.2}.
Hence Equation \eqref{eq5.3} has the same form as \eqref{P-6}.
It is clear that $\ell$ is bounded in $\Sigma$. Therefore the same proof given in Section 2 yields a solution $w$, together with a constant $c$. Then $u=u_0+w$, together with $\a=\a_0+c$, is a solution of \eqref{eq5.2} with a quadratic growth which implies
\begin{equation} \label{eq5.4}
\lim_{|x|\to \infty} \frac {u(x)} {|x|} =\infty.
\end{equation}

A complete Riemannian manifold $(\S, g)$ has a {\it intrinsic polynomial volume growth} if there exist $\beta >0$ and $C>0$ such that $vol(B_r) \le Cr^\beta$ for any $r>0$, where $B_r$ is the geometric ball of radius $r$ with respect to the metric $g$. 
It is clear that for submanifolds  the extrinsic polynomial volume condition is stronger than the intrinsic one.

\begin{lemma}
	\label{lem3.1a}   Let $\Sigma$ be an $n$-dimensional complete, non-compact submanifold in $\R^{n+m}$ with   polynomial 
	 growth as in Definition \ref{Def1} and $f:\Sigma\to \R$ a non-negative function with compact support.
	For $f_\epsilon$ $(\epsilon>0)$ there exists a solution $u$ of \eqref{eq5.2} with \eqref{eq5.4} and
	\begin{equation}\label{div2}
	\int_{\Sigma} \div(f_\epsilon \n u )=0.
	\end{equation}
\end{lemma}
\begin{proof}
	As discussed above we have a solution 
	$u$ of \eqref{eq5.2} with \eqref{eq5.4}. Since $\Sigma$
	has the intrinsic polynomial 
	volume growth and its mean curvature vector has at most polynomial growth, the same proof for Lemma \ref{lem2.4} works here to get \eqref{div2}. 
\end{proof}

For this $u$ we define as in \cite{Brendle22}
\[
U:=\{(x,y) \, |\, x\in \Sigma, y\in T_x^\perp \Sigma, \n ^ 2 u(x)-\<II(x), y\> \ge 0\}
\] and a map $\Phi:U\to \R^{n+m}$
by
\[
\Phi(x,y) =\n u(x) +y,
\]
where $T_x\Sigma$ is the normal space of
$\Sigma$ at $x$ and $II(x)$ is the second fundamental form.
The growth property \eqref{eq5.4}, together with the proof of Lemma 2.1 in
\cite{Brendle22}, implies that
\begin{equation}\label{RRR}
\Phi(U)=\R^{n+m}.\end{equation}

\begin{lemma}
	\label{lem3.2} For any $(x,y) \in U$ we have
	\begin{eqnarray} \label{eq42}
	0\le \det D\Phi(x,y)&=&\det (\n^2 u(x)-\<II(x), y\>) \\&\le &f(x) e^{\frac {|\Phi(x,y)|^2}4-\frac {|2H(x)+y|^2} 4+\left|\frac 12 x^\perp + H\right|^2+\a -n}\nonumber
	\end{eqnarray}
\end{lemma}

\begin{proof}The equality has been proved in \cite{Brendle21}. We show the last inequality.
	From Equation \eqref{eq5.2}
 we have
 \begin{eqnarray*}
\Delta u(x)-\<H(x), y\> &=&\ds\vs
\log f_\e(x) -|\n \log f_\e(x)|^2-\n \log f_\e(x) \cdot \n u(x)\\ &&  -|H(x)|^2 +\left|\frac 12 x^\perp+H(x)\right|^2 +\a-\<H(x), y\>\\
&\le & \ds \log f_\e(x) +\frac {|\n u(x)|^2+|y|^2}4 -\frac {|2H(x)+y|^2} 4\\&&+\left|\frac 12 x^\perp+H(x)\right|^2 +\a.
 \end{eqnarray*}
We then have
\begin{eqnarray*}
   0&\le& \det (\n^2 u(x)-\<II(x), y\>) 
\le e^{\Delta u-\<H(x), y\>-n}
\\&\le& f_\e(x) e^{\frac {|\Phi(x,y)|^2}4-\frac {|2H(x)+y|^2}4+\left|\frac 12 x^\perp + H\right|^2+\a-n}, 
\end{eqnarray*}
since $|\Phi(x,y)|^2= |\n u|^2  +|y|^2$.
\end{proof}

Now we can prove Theorem \ref{thm2}. For the convenience of the reader we restate Theorem \ref{thm2} here.

\begin{theorem}\label{thm_sub}
	Let $\Sigma$ be an $n$-dimensional complete, non-compact submanifold in $\R^{n+m}$ with polynomial 
	volume growth. Then for any positive function $f$ on $\Sigma$ with $\int_\Sigma f dvol =1$, there holds
	\begin{equation}
	\label{LSI4Sub0}
	\begin{array}{rcl} &&\ds\int_\Sigma \left\{ \frac {|\n f|^2} f-f \log f +|H|^2 f\right\}dvol\\&\ge& \ds \vs
 n+\frac n 2 \log(4\pi)  +\int_\Sigma f \left|\frac 12 x^\perp + H\right|^2 dvol -\log \int_{\Sigma}f (x) e^{\left|\frac 12 x^\perp + H\right|^2} dvol.
	\end{array}    
	\end{equation}
\end{theorem}

\begin{proof} We first show \eqref{LSI4Sub0} for $f_\e$. From \eqref{RRR} and \eqref{eq42} we have
	\begin{equation}
	\begin{aligned}
	1&=(4\pi)^{-\frac{n+m}{2}}\int_{\mathbb R^{n+m}}e^{-\frac{|\xi|^2}{4}}d\xi\\
	&\leq (4\pi)^{-\frac{n+m}{2}} \int_\Sigma \left(\int_{T^ \perp_x\Sigma}
	e^{-\frac{|\Phi(x,y)|^2}{4}}|\det D\Phi(x,y)|1_{U}(x,y)\, dy \right) dvol(x)\\
	&\leq (4\pi)^{-\frac{n+m}{2}} \int_\Sigma f_\e(x)\left(\int_{T^ \perp_x\Sigma}e^{-\frac {|2H(x)+y|^2} 4+\left|\frac 12 x^\perp + H\right|^2+\a -n}
	dy \right) dvol(x)\\
	&= (4\pi)^{-\frac{n}{2}} \int_\Sigma f_\e(x)e^{\left|\frac 12 x^\perp + H\right|^2+\a -n}
	dvol(x).
	\end{aligned}
	\end{equation} 
	It follows
	that
	\[
	\a\ge n+\frac n 2 \log (4\pi) -\log\int_\S f_\e(x) e^{\left|\frac 12 x^\perp + H\right|^2}.
	\]
	Integrating Equation  \eqref{eq5.2} over $\Sigma$, together with \eqref{div2}, we have
\[\begin{array}{rcl}
  &&\ds\vs \int_\Sigma \left\{
  -f_\e  \log f_\e + f_\e|\n \log f_\e|^2 - \left|\frac 12 x^\perp + H\right|^2 f_\e+|H|^2 f_\e  \right\} dvol \\
  &=&  \ds\vs \a \int_\Sigma f_\e  dvol =\a  \\
  &\ge& \ds  n+\frac n 2 \log (4\pi) -\log\int_\S f_\e(x) e^{\left|\frac 12 x^\perp + H\right|^2}.
  \end{array}
\]
By letting $\e\to 0$, we get the assertion.
\end{proof}

\begin{remark}
    By Jensen's inequality we know 
\[
\int_\Sigma f \left|\frac 12 x^\perp + H\right|^2 dvol -\log \int_{\Sigma}f (x) e^{\left|\frac 12 x^\perp + H\right|^2} dvol\le 0.
\]
Hence \eqref{LSI4Sub0}
is weaker than
\begin{equation}
\label{LSI_Sub_B}
\begin{array}{rcl} \ds\int_\Sigma\left\{ \frac {|\n f|^2} f-f \log f +|H|^2 f\right\}dvol \ge \ds\vs n+\frac n 2 \log(4\pi),
\end{array}    
\end{equation}
which was proved by Brendle for compact submanifolds without boundary. It would be interesting to ask if \eqref{LSI_Sub_B} is also true for non-compact submanifolds with the polynomial growth condition.
\end{remark}

It is interesting to see that in the important case, where  $\Sigma$ is a self-shrinker, ie.
\[
\frac 12 x^\perp+H=0,
\]
\eqref{LSI4Sub0} is the same as \eqref{LSI_Sub_B}. Now we prove Theorem \ref{thm1}.
\begin{theorem}
	Let $\Sigma$ be an $n$ dimensional complete, non-compact properly embedded self-shrinker in $\R^{n+m}$. Then for any positive function $f$ on $\Sigma$ with $\int_\Sigma f dvol =1$ \eqref{LSI_Sub_B} holds. Equivalently, for 
	any positive function $\varphi$ with
	$\int_\S \varphi d\gamma=1$ $(d\gamma =\gamma_0 dvol)$, there holds
 \begin{equation}\label{LSI_soliton}
	\int_\S \varphi \log\varphi \,d\gamma \le \int_\S \frac {|\n \varphi|^2}\varphi \,d\gamma.
	\end{equation}
\end{theorem}
\begin{proof}
	From \cite{Ding-Xin} and \cite{Cheng-Zhou} we know that for a self-shrinker the properly embeddedness is equivalent to the extrinsic polynomial volume growth. Moreover the self-shrinker condition clearly implies that the norm of $H$ has linear growth. Hence Theorem \ref{thm_sub} implies \eqref{LSI_Sub_B}, since $\S$ is a self-shrinker.   {As computed in \cite{Brendle22}, we see
   \begin{eqnarray*}
        &&\frac {|\n f|^2} f-f \log f +|H|^2 f -(n+\frac n 2 \log (4\pi))f
        \\&=&   \frac {|\n \varphi|^2}\varphi \g_0- \varphi \log\varphi  \g_0 +\left|\frac 12 x^\perp + H\right|^2 f-{\rm div}_\S(fx^T).  
   \end{eqnarray*}
   Integrating over $(\S, dvol)$ and using that $\S$ is a self-shrinker and $\int_\S f=1$, we see that \eqref{LSI_Sub_B} is equivalent to
    \begin{eqnarray*}
        &&   \int_\S \frac {|\n \varphi|^2}\varphi d\g- \varphi \log\varphi  d\g \ge \int_\S{\rm div}_\S(fx^T) dvol.  
   \end{eqnarray*}
 Note that for each $R>0$ and $B_R\subset \R^{n+m}$, 
   \begin{eqnarray*}
   \int_{\S\cap B_R}{\rm div}_\S(fx^T) dvol=
   \int_{\S\cap \p B_R}f|x^T| dA \ge 0.
    \end{eqnarray*}
 \eqref{LSI_soliton} follows by letting $R\to \infty$.
   }
\end{proof}



\begin{remark} In Theorem \ref{thm1} without the normalization $\int_\S \varphi d\gamma=1$, \eqref{eq0.3} is written as
\begin{equation}
	\label{eq_a1}
	\int_\S \frac {|\n \varphi|^2}\varphi \,d\gamma \ge \int_\S \varphi \log\varphi \,d\gamma -\int_\S \varphi d\gamma \log \int_\S \varphi d\gamma.
	\end{equation}
	In general $d\g=\gamma_0 d vol$ may not be a probability measure on $\S$. In fact,
	$\int_\Sigma 1 d \gamma=\l(\Sigma)$, the entropy of $\Sigma$, since $\S$ is self-shrinker. See \cite{Colding-Minicozzi}. If we use the probability measure, $d\tilde \gamma= \l(\Sigma)^{-1} d\gamma$, then \eqref{eq0.2} can be written as 
	\begin{equation}
	\label{eq_a2}
	\int_\S \frac {|\n \psi|^2}\psi \,d \tilde \gamma - \int_\S \psi\log\psi \,d\tilde \gamma 
 \ge -\log \l(\Sigma)  
	\end{equation}
	for any non-negative $\psi$
 with $\int_\S \psi d\tilde \gamma =1$. For such a self-shrinker $\Sigma$, we know $\l(\Sigma)\ge 1$. See \cite{Colding-Minicozzi}.
We define an invariant by \[\mu_\S:=\inf _ {0\not \equiv \psi\ge 0, \int_\S \psi d\tilde \gamma =1}  \left\{	 \int_\Sigma \frac {|\n \psi|^2}\psi  \,d\tilde \gamma- 	\int _\Sigma \psi \log \psi \,d\tilde \gamma  \right\}.\]\eqref{eq_a2} means that\begin{equation}\mu_\S\ge -\log \l(\S).\end{equation}It would be interesting to discuss the relationship between   $\mu_\S$ and $ -\log \l(\S)$.

\end{remark}



\begin{remark}
    \label{rmk3.6}
    Let $f_\tau = (4\pi\tau)^{\frac n 2} e^{-\frac{|x|^2}{4\tau}}$ for $\tau>0$. It is easy to check as above that $u_0=\frac 12 |x|^2$ satisfies
    \[
    f_\tau ^{-1}\div (f_\tau \n u_0)=\log f_\tau -\tau |\n \log f_\tau|^2 -\tau |H|^2 +\tau |H+\frac {x^\bot} {2\tau} |^2 +\a_\tau,
    \]
    where $\a_\tau=n+\frac n 2 \log(4\pi\tau).$ The same proof yields
    the following inequality
\begin{equation}
	\label{hin}
	\begin{array}{rcl} &&\ds\int_\Sigma \left\{\tau  \frac {|\n f|^2} f-f \log f +\tau |H|^2 f\right\}dvol\\&\ge& \ds\vs n+\frac n 2 \log(4\pi\tau)  \ds +\int_\Sigma \tau \left|\frac 1{2\tau} x^\perp + H\right|^2  f\,dvol -\log \int_{\Sigma} e^{\tau \left|\frac 1{2\tau} x^\perp + H\right|^2} f\, dvol,
	\end{array}    
	\end{equation}
which implies that
\begin{equation}
	\label{Last2}
	\begin{array}{rcl} \ds\int_\Sigma \left\{\tau  \frac {|\n f|^2} f-f \log f +\tau |H|^2 f\right\}dvol-\frac n2 \log(4\pi\tau)\ge 
   n,
	\end{array}    
	\end{equation}
 for non-negative function $f$ with $\int_\S f=1$,
 if $\S$ is a properly embedded self-shrinker. 
 \end{remark}

From above remarks one  can compare the 
logarithmic Sobolev inequalities for self-shrinkers of mean curvature flow and the ones for gradient solitons of Ricci flow obtained by Carrillo-Ni \cite{C-Ni} and Hein-Naber \cite{Hein-Naber}, where the non-negativity of the Ricci curvature
or the $C(K,\infty)$ condition play a crucial role.

\

\noindent{\it Acknowledgment.} A part of the paper was carried out while the first  and the third named authors were visiting McGill university. They would like to thank Professor Pengfei Guan for his warm hospitality.

After submitting the paper for publication, we learnt that the similar inequality was also proved by Z. Balogh and  A Kristály in
\cite{BK24} using optimal transport.

\end{document}